\documentclass[A4paper,reqno,oneside,12pt]{amsart}
\usepackage{fullpage}
\usepackage{setspace}
\usepackage{amsmath, amssymb, amsfonts, amsthm, bbm, mathtools, mathrsfs}
\usepackage{tikz}
\usepackage{subcaption}
\usepackage{verbatim,enumitem,graphics}
\usepackage{xcolor}
\usepackage[backref=page]{hyperref}
\usepackage{cleveref}
\hypersetup{%
    colorlinks,
    linkcolor={red!50!black},
    citecolor={green!50!black},
    urlcolor={blue!50!black}
}
\usepackage{dsfont}
\usepackage{lineno}
\usepackage[normalem]{ulem}

\allowdisplaybreaks

\theoremstyle{plain}
\newtheorem{theorem}{Theorem}[section]
\newtheorem*{theorem*}{Theorem}
\newtheorem{lemma}[theorem]{Lemma}

\newtheorem{facttt}[theorem]{Fact}

\theoremstyle{definition}

\newtheorem*{remark*}{Remark}
\newtheorem{definition}[theorem]{Definition}

\newcommand{\N}{\mathbb{N}}
\newcommand{\Z}{\mathbb{Z}}

\newcommand{\R}{\mathbb{R}}

\newcommand{\paren}[1]{\left( #1 \right)}

\newcommand{\setcond}[2]{\left\{ #1 \;\middle\vert\; #2 \right\}}

\begin{document}

\setstretch{1.27}

\title{On exponential Freiman dimension}

\author{Jeck Lim}
\address{University of Oxford, Radcliffe Observatory, Andrew Wiles Building, Woodstock Rd, Oxford OX2 6GG, United Kingdom}
\email{jeck.lim@maths.ox.ac.uk}

\author{Akshat Mudgal}
\address{Mathematics Institute, Zeeman Building, University of Warwick, Coventry CV4 7AL, UK.}
\email{Akshat.Mudgal@warwick.ac.uk}

\author{Cosmin Pohoata}
\address{Department of Mathematics, Emory University, Atlanta, GA, 30322, USA}
\email{cosmin.pohoata@emory.edu}

\author{Xuancheng Shao}
\address{Department of Mathematics, University of Kentucky, 715 Patterson Office Tower, Lexington, KY 40506, USA}
\email{xuancheng.shao@uky.edu}


\begin{abstract}
\vspace{+5mm}
 The exponential Freiman dimension of a finite set $A \subset \mathbb{R}^{m}$, introduced by Green and Tao in 2006, represents the largest positive integer $d$ for which $A$ contains the vertices of a non-degenerate $d$-dimensional parallelepiped. 
 
 For every $d \geq 1$, we precisely determine the largest constant $C_{d}>0$ (exponential in $d$) for which 
$$|A+A| \geq C_{d}|A| - O_{d}(1)$$
holds for all sets $A$ with exponential Freiman dimension $d$.
\end{abstract}
\maketitle
\section{Introduction}

A central theme in additive combinatorics suggests that in any abelian group, finite sets $A$ with  small doubling
\[
  \sigma(A) \;:=\; \frac{|A+A|}{|A|}
\]
must exhibit strong additive structure.  Here, $A+ A = \{ a+ a' : a,a' \in A\}$ denotes the sumset of $A$. Classical results of Freiman and their later refinements show that such sets are efficiently contained in
generalized arithmetic progressions or subspaces with bounded dimension; see, for example, Freiman's monograph
\cite{FreimanBook} and the book of Tao and Vu~\cite{TaoVu}.

When the abelian group above has no torsion, these questions are often related to studying the following problem: given a finite set $A \subset \mathbb{R}^m$, what type of geometric conditions must $A$ satisfy so as to ensure that the doubling $\sigma(A)$ grows rapidly, see \cite{Bilu, Chang}.
A basic instance of this is \emph{Freiman's lemma}, which states that if a finite set \(A\subset\R^m\) contains the vertices of a non-degenerate \(d\)-simplex, then
\[
  |A+A| \geq (d+1)|A| - \tfrac12 d(d+1) \geq d|A|/2.
\]
A non-degenerate \(d\)-simplex is a set $P \;=\; \{v_0,\,v_0+v_1,\dots,v_0+v_d\}$
with \(v_1,\dots,v_d\) linearly independent. So the mere presence of a simplex of dimension \(d\) forces the additive doubling $\sigma(A)$ to grow linearly in $d$.

This can be compared with the continuous setting where given a non-empty, compact set $\mathcal{A} \subset \mathbb{R}^d$,  one can apply the Brunn--Minkowski inequality to deduce that $\mu(\mathcal{A} + \mathcal{A}) \geq 2^d \mu(\mathcal{A})$, with $\mu$ denoting the Lebesgue measure in $\mathbb{R}^d$.  Thus it is natural to ask what local conditions can be prescribed for a finite set $A \subseteq \mathbb{R}^d$ to ensure that $\sigma(A)$ grows exponentially in $d$. 

In this direction, Green and Tao \cite{GTcompressions} introduced the notion of \emph{exponential Freiman dimension}. Thus, given a finite, non-empty set $A \subset \mathbb{R}^m$, we define its exponential Freiman dimension  to be the the largest positive integer $d$ for which $A$ contains the vertices of a non-degenerate $d$-dimensional parallelepiped
\[
  P \;=\; v_0 + \{0,1\} \cdot v_1 + \cdots + \{0,1\} \cdot v_d
\]
with \(v_1,\dots,v_d\) linearly independent. Green and Tao \cite{GTcompressions} proved that if $A \subset \mathbb{R}^m$ has exponential Freiman dimension $d$, then 
\begin{equation} \label{sqrt2}
  |A+A| \geq (\sqrt{2})^{d} \,|A|.
  \end{equation}
Such inequalities, while being of independent interest, have also had some exciting recent applications to sum-product theory, see work of P\'{a}lv\"{o}lgyi--Zhelezov ~ \cite{PZ2020}. This circle of ideas has been significantly generalized in \cite{EntropyPFR}, the latter leading to the breakthrough work of Gowers--Green--Manners--Tao \cite{PFR} on the  polynomial Freiman--Ruzsa conjecture over $\mathbb{F}_2^n$, and subsequently yielding previously inaccessible sum-product estimates over $\mathbb{R}$.

In \cite{GTcompressions}, Green and Tao noted that the inequality \eqref{sqrt2} is somewhat close to sharp: for example, if $A = \{0,1\}^d\subset\R^d$ one has $|A+A| = 3^d = (3/2)^d|A|$. More specifically, given a positive integer $d$, if we denote by $C_d>0$ the largest constant such that
\begin{equation}\label{eq:Cd-def}
  |A+A| \;\ge\; C_d\,|A| - O_d(1)
\end{equation}
holds for all finite sets $A \subset \mathbb{R}^{m}$ with exponential Freiman dimension $d$, then Green and Tao established that 
$$(\sqrt{2})^{d} \leq C_{d} \leq 2(3/2)^{d-1},$$
with the upper bound following from considering the set $A = \{0,1\}^{d-1} \times \{0,1,\dots, N\}$.

Our main result closes the above exponential gap by determining the exact value of \(C_d\) for every $d \geq 1$.  
\begin{theorem}\label{thm:Cd-main}
For every $d \geq 1$ and every finite set $A\subset \R^m$ with exponential Freiman dimension $d$, we have 
\begin{equation} \label{eq:main-tight}
  |A+A|\geq C_d(|A|-2^d)+3^d,
\end{equation}
where
\[
  C_d \;=\;
  \begin{cases}
    2\bigl(\tfrac32\bigr)^{d-1}, & d \le 5,\\[4pt]
    2^{d/2+1} - 1, & d \ge 6 \text{ even},\\[4pt]
    3\cdot 2^{(d-1)/2} - \tfrac32, & d \ge 7 \text{ odd}.
  \end{cases}
\]
\end{theorem}

A perhaps surprising aspect of Theorem~\ref{thm:Cd-main} is  that the natural upper bound example $A = \{0,1\}^{d-1} \times \{0,1,2,\dots, N\}$ which satisfies $|A+A| = 2 (3/2)^{d-1}|A| + O_d(1)$ is only optimal when $d \leq 5$. When $d >5$, there is a significantly different family of examples that extremize the doubling $\sigma(A)$.
Writing \([0,N] := \{0,1,\dots,N\}\)
for any \(N\in\N\), we describe these examples below.
\begin{align}
  A &= \{0,1\}^{d-1} \times [0,N], \label{eq:example-0.1}\\
  A &= \{0,1\}^d \,\cup\, \bigl(\{0\}^{d/2} \times [0,N]^{d/2}\bigr)
      &&\text{for \(d\) even}, \label{eq:example-0.2}\\
  A &= \{0,1\}^d \,\cup\, \bigl(\{0\}^{(d-1)/2} \times \{0,1\}\times [0,N]^{(d-1)/2}\bigr)
      &&\text{for \(d\) odd}. \label{eq:example-0.3}
\end{align}
Example~\eqref{eq:example-0.1} satisfies $|A| = 2^{d-1}(N+1)$, and $A+A = \{0,1,2\}^{\,d-1}\times[0,2N]$, so 
$$|A+A| = 3^{\,d-1}(2N+1)= 2(3/2)^{d-1}(|A|-2^d)+3^d.$$
Example~\eqref{eq:example-0.2} satisfies $|A| = (N+1)^{d/2} + 2^d - 2^{d/2}$ and 
\[ A+ A = \{0,1,2\}^d \cup (\{0,1\}^{d/2} \times [0,N+1]^{d/2}) \cup (\{0\}^{d/2} \times [0, 2N]^{d/2}),  \]
whence, 
\[ \frac{|A+A|}{|A|} = \frac{2^{d/2}(N+2)^{d/2} + (2N+1)^{d/2} - (N+2)^{d/2} + O_d(1)}{(N+1)^{d/2} + O_d(1)} \rightarrow 2 \cdot 2^{d/2}-1 \]
as $N\rightarrow\infty$.
Finally, Example~\eqref{eq:example-0.3} satisfies  $|A| = 2(N+1)^{(d-1)/2} + 2^d - 2^{(d+1)/2}$ and
\begin{align*}
    A+A &= \{0,1,2\}^d \cup (\{0,1\}^{(d-1)/2}\times \{0,1,2\}\times [0,N+1]^{(d-1)/2}) \\
    & \quad \cup (\{0\}^{(d-1)/2}\times \{0,1,2\}\times [0,2N]^{(d-1)/2}),
\end{align*}
and so, we get that 
\begin{align*}
    \frac{|A+A|}{|A|} &= \frac{3\cdot 2^{(d-1)/2}(N+2)^{(d-1)/2} + 3\cdot (2N+1)^{(d-1)/2} - 3\cdot (N+2)^{(d-1)/2} + O_d(1)}{2(N+1)^{(d-1)/2} + O_d(1)}\\
    &\rightarrow 3\cdot 2^{(d-1)/2} - \tfrac32
\end{align*}
as $N\rightarrow\infty$.

Taking the minimum of these three upper bounds for each \(d\) recovers exactly
the values in Theorem~\ref{thm:Cd-main}. Furthermore, \eqref{eq:main-tight} is exactly tight for $d\leq 5$. Interestingly, for sufficiently large $d$, since $C_d<(3/2)^d$, Theorem \ref{thm:Cd-main} yields the bound
\[ |A+A|\geq C_d|A|+(3^d-2^dC_d), \]
wherein the error term is actually positive.

The main difficulty will be to show that $C_{d}$ is at least the indicated values above. Before we move on to the proof of Theorem \ref{thm:Cd-main} in \S2, let us start by discussing some of the high-level ideas behind it. 

\medskip

{\bf{Proof ideas}}. By applying a suitable linear transformation and translation we may assume
that $P=\left\{0,1\right\}^{d} \subseteq A \subset \R^{d}$. By further applying translations
and coordinate-wise compressions, we may also assume that $A$ is a \emph{down-set} in $\N_0^d$, that is, for any \(u\in A\) and any $v \in \mathbb{Z}^d$ satisfying the inequality \(0\le v\le u\) coordinate-wise, we have \(v\in A\).  In particular, the
assumption \(\{0,1\}^d \subseteq A\) says that the ``bottom layer'' of \(A\) is a
complete discrete hypercube.  We then partition \(A\) according to which
coordinates lie below or above a fixed threshold, and obtain a parallel
partition of the sumset \(B := A+A\).  We will then employ a discrete Brunn--Minkowski type inequality, due to Green and Tao \cite{GTcompressions}, to get lower bounds on the sizes of the blocks of \(B\)
in terms of the blocks of \(A\).

After a suitable renormalisation, this reduces Theorem~\ref{thm:Cd-main} to a
purely analytic inequality for families of non-negative real numbers
\(\{x_u\}_{u\in\{0,1\}^k}\) and \(\{y_w\}_{w\in\{0,1,2\}^k}\), indexed by slices
of the discrete (Hamming) hypercube and hypergrid, subject to two constraints:
a collection of inequalities of the form \(y_{u+v}\ge x_u+x_v\) arising from the sumset structure, and a monotonicity condition reflecting the downset
structure. The monotonicity condition is not strictly necessary, but it is convenient to have. In order to state this, we will require one further piece of notation, and so, given vectors $u=(u_1, \dots, u_d)$ and $v= (v_1, \dots, v_d)$ in $\mathbb{R}^d$, we say that $v\succeq u$ if $v_i \geq u_i$ for every $1 \leq i \leq d$. 

\begin{theorem} \label{lem:mainineq}
    Let $0\leq k<d$ be integers. Let $\{x_u: u\in \{0,1\}^k\}$ and $\{y_u: u\in \{0,1,2\}^k\}$ be sets of non-negative real numbers such that $y_{u+v}\geq x_u+x_v$ for each $u,v\in \{0,1\}^k$. Assume further that $x_u\geq x_v$ if $v\succeq u$. Then we have
    \begin{equation} \label{anineq}
        \sum_{u\in \{0,1,2\}^k} y_u^{d-k}\geq C_d \sum_{u\in \{0,1\}^k} x_u^{d-k},
    \end{equation}
    where $C_d$ is as in Theorem \ref{thm:Cd-main}, namely
    \[
  C_d \;=\;
  \begin{cases}
    2\bigl(\tfrac32\bigr)^{d-1}, & d \le 5,\\[4pt]
    2^{d/2+1} - 1, & d \ge 6 \text{ even},\\[4pt]
    3\cdot 2^{(d-1)/2} - \tfrac32, & d \ge 7 \text{ odd}.
  \end{cases}
\]
\end{theorem}

We remark that the constant $C_d$ in Theorem~\ref{thm:Cd-main} is tight for at least one value of $k$ for each $d$. In particular, for $d\leq 5$, tightness occurs at $k=d-1$, with $x_u=1$ for all $u$. For even $d\geq 6$, tightness occurs at $k=d/2$ with $x_0=1$ and $x_u=0$ for all $u\neq 0$. Finally, for odd $d\geq 7$, tightness occurs at $k=(d+1)/2$ with $x_0=x_{e_1}=1$ and $x_u=0$ for all other $u$. These equality cases naturally result in the constructions given in \eqref{eq:example-0.1}, \eqref{eq:example-0.2} and \eqref{eq:example-0.3}.

Theorem \ref{lem:mainineq} is proved by a
combination of averaging arguments, case distinctions in small dimensions, and compression properties that allow us to redistribute weight from parts with excess weight to parts with less weight.  The
dimension-dependent constants that arise along the way match precisely the
values of \(C_d\) asserted in Theorem~\ref{thm:Cd-main}.

\medskip

{\bf{Organization of the paper}}. In \S2 we discuss various prerequisites concerning compressions and prove how  Theorem~\ref{thm:Cd-main} follows from Theorem \ref{lem:mainineq}. We lay some preliminary brickwork for proving our main analytic inequality in \S3, which already turns out to be sufficient to deal with the small-dimensional cases \(d \in \{3,4,5\}\). The case when $d \geq 6$ is even is resolved in \S4, while the case when $d \geq 7$ is odd is analysed in \S5.  The odd case is slightly more involved compared to the even case, since the corresponding doubling constant $C_d$ is larger and the construction \eqref{eq:example-0.3} attaining ``equality'' is more complicated. In \S6, we will discuss a few natural future directions.

\medskip

{\bf{Acknowledgments.}} The authors would like to thank Ben Green for useful discussions. J.L. was supported by ERC Advanced Grant 883810. A.M. was supported by Leverhulme early career fellowship ECF-2025-148. C.P. was supported by NSF grant DMS-2246659. X.S. was supported by NSF grant DMS-2452462.

\medskip

\section{Compressions and Brunn-Minkowski}

By applying a suitable linear transformation and translation we may first assume
that $P=\left\{0,1\right\}^{d} \subset A \subset \R^{m}$. Moreover, upon applying a suitable Freiman homomorphisms from $\R^m \to \R^d$ if necessary, we are able to assume that $m = d$.

\subsection*{Coordinate compressions and down-sets}

Throughout the paper, we write $\N_{0}$ for the set of non-negative integers and we will be considering
$\N_{0}^{d}$ with its standard coordinatewise partial order
\[
x \preceq y \quad\Longleftrightarrow\quad x_i \le y_i \;\text{ for all } 1 \le i \le d.
\]

\begin{definition}[Down-sets]
A subset $X \subset \R^{d}$ is called an \emph{$i$-down-set} (in the $i$-th coordinate) if whenever
$(x_1,\dots,x_d) \in X$, then $x_i\in \N_0$ and for any smaller non-negative integer $y_i \le x_i$, the vector
\[
(x_1,\dots,x_{i-1},y_i,x_{i+1},\dots,x_d)
\]
also lies in $X$. We say that $X$ is a \emph{down-set} if it is an $i$-down-set for every $1 \le i \le d$.
\end{definition}

Thus a down-set is simply a set in $\N_0^d$ that is closed under moving points ``downwards'' in any coordinate.  
Typical examples include the discrete box $[L_1]\times \cdots \times [L_d]$ (where $[L] := \{0,1,\dots,L-1\}$)
and, more generally, all subsets of such a box that are closed under the partial order $\preceq$. One reason down-sets are convenient is that they behave well under Minkowski sums: if $X,Y \subset \R^{d}$ are down-sets, then so is $X+Y$. This makes them a natural class of sets on which to study sumsets. The passage from an arbitrary finite set $A \subset \R^{d}$ to a down-set in $\N_0^d$ is achieved by
\emph{coordinate compressions}. Fix a coordinate $i \in \{1,\dots,d\}$, and for each ``fiber''
\[
H_{\mathbf{u}} := \{x \in \R^{d} : x_j = u_j \text{ for all } j\neq i\},
\]
we consider the intersection $A \cap H_{\mathbf{u}} \subset H_{\mathbf{u}}$, which is just a finite subset of
$\R$ when we identify $H_{\mathbf{u}}$ with the $i$-th axis. The $i$-compression of $A$ is obtained by
moving the points of $A$ in each such fiber as far down as possible, while preserving the size of the
intersection.

\begin{definition}[$i$-compression]
Let $A \subset \R^{d}$ be finite and fix $1 \le i \le d$. For each $\mathbf{u}\in\R^{d}$, set
\[
m(\mathbf{u}) := |A \cap H_{\mathbf{u}}|.
\]
We define the \emph{$i$-compression} $C_i(A)$ of $A$ by
\[
C_i(A)
 := \Bigl\{
(u_1,\dots,u_{i-1},t,u_{i+1},\dots,u_d) : 
t\in \N_0, 0 \le t < m(\mathbf{u}) \text{ for some } \mathbf{u}
\Bigr\},
\]
that is, in each fiber $H_{\mathbf{u}}$, we replace $A \cap H_{\mathbf{u}}$ by the initial segment
$\{0,1,\dots,m(\mathbf{u})-1\}$.
\end{definition}

By construction, $C_i(A)$ has the same cardinality as $A$ and is an $i$-down-set. In particular, if $A$
was already a $j$-down-set for some $j\neq i$, then $C_i(A)$ remains a $j$-down-set. Iterating over
$i=1,\dots,d$ we obtain a genuine down-set:

\begin{facttt}
For any finite $A \subset \R^{d}$, the iterated compression
\[
C(A) := C_d \circ C_{d-1} \circ \cdots \circ C_1(A)
\]
is a down-set with $|C(A)| = |A|$.
\end{facttt}

The key feature of these compressions, first introduced by Kleitman \cite{Kleitman}, and further exploited in the works of Bollobás--Leader \cite{BollobasLeader96}, Gardner-Gronchi \cite{GardnerGronchi01} and Green--Tao \cite{GTcompressions},
is that they tend to \emph{decrease} sumsets.

\begin{facttt}
For any finite $A \subset \R^{d}$ and any $1 \leq i \leq d$, we have
$$|C_i(A) + C_i(A)| \;\le\; |A+A|.$$
In particular, $|C(A)+C(A)| \leq |A+A|$. 
\end{facttt}
When studying \emph{lower} bounds on $|A+A|$, we may thus compress $A$ in each coordinate until $A$ becomes a down-set in $\N_0^d$, and then work this more structured set instead. We shall proceed this way in our setting also. Namely, after applying the appropriate coordinate-wise compressions, let us now assume that our set $A$ is a down-set satisfying $\left\{0,1\right\}^{d} \subset A \subset \N_0^{d}$. 

\subsection*{Block decomposition}

For a subset $S\subseteq [d]$, let $A_S$ be the set of elements $(a_1,\ldots,a_d)\in A$ such that $a_i<2$ if $i\in S$ and $a_i\geq 2$ if $i\notin S$. Then we have a partition $A=\bigcup_{S\subseteq [d]} A_S$. In particular, $A_{[d]}=\{0,1\}^d$. We can further partition $A_S$ into $2^{|S|}$ parts $A_S(v)$ for each $v\in \{0,1\}^S$, where $A_S(v)\subset \Z^{[d]\setminus S}$ consists of the points $(a_i)_{i\in [d]\setminus S}$ if there exists $(a_1,\ldots,a_d)\in A_S$ with $a_i=v_i$ for each $i\in S$.

Let $B=A+A$. For each subset $S\subseteq [d]$, let $B_S$ be the set of elements $(b_1,\ldots,b_d)\in B$ such that $b_i<3$ if $i\in S$ and $b_i\geq 3$ if $i\notin S$. Then we have a partition $B=\bigcup_{S\subseteq [d]} B_S$. In particular, $B_{[d]}=\{0,1,2\}^d$. We can further partition $B_S$ into $3^{|S|}$ parts $B_S(v)$ for each $v\in \{0,1,2\}^S$, where $B_S(v)\subset \Z^{[d]\setminus S}$ consists of the points $(b_i)_{i\in [d]\setminus S}$ if there exists $(b_1,\ldots,b_d)\in B_S$ with $b_i=v_i$ for each $i\in S$.

The aim is to show that for each $S\subsetneq [d]$, we have $|B_S|\geq C_d |A_S|$. Then it follows that
\begin{align*}
    |B| & = \sum_{S\subsetneq [d]}|B_S|+|B_{[d]}|\geq C_d \sum_{S\subsetneq [d]}|A_S|+3^d = C_d(|A|-2^d)+3^d.
\end{align*}

From now on, we fix $S\subsetneq [d]$, which we may assume to be $S=[k]$ without loss of generality, for some $0\leq k<d$.

\begin{lemma} \label{lem:fibersum}
    For any $u,v\in \{0,1\}^k$, we have $B_S(u+v)\supseteq A_S(u)+A_S(v)+\{-1,0\}^{d-k}$. Furthermore, if say $A_S(u)=\emptyset$, then $B_S(u+v)\supseteq A_S(v)+\{1\}^{d-k}$.
\end{lemma}

\begin{proof}
    Since $A$ is a downset, we get that $B = A+A$ is a downset. Let $u, v \in \{0,1\}^k$. By our definition of $B_S(u+v)$, we further get that $B_S(u+v) - z$ is a downset in $\N_0^{d-k}$, where $z = (3,\dots, 3) \in \N_0^{d-k}$.  Now trivially we see that $B_S(u+v) \supseteq A_S(u) + A_S(v)$, and so, $B_S(u+v) - z \supseteq A_S(u) + A_S(v) - z$. On the other hand, since for any $x \in A_S(u), A_S(v)$ we have $x \in [2, \infty)^{d-k}$, we see that any element in $A_S(u) + A_S(v)  -z \in [1, \infty)^{d-k}$. Finally since $B_S(u+v) - z$ is a downset in $\N_0^{d-k}$, the preceding discussion implies that we must have
    \[ B_S(u+v) - z \supseteq A_S(u) + A_S(v) - z - \{0,1\}^{d-k}.\]
    This simplifies to give us the desired conclusion in case when $A_s(u) \neq \emptyset$.

    If $A_S(u) = \emptyset$, then upon letting $w=v\times \{1\}^{d-k}\in A$, we see that $B_S\supseteq A_S+w$, which consequently implies that $B_S(u+v)\supseteq A_S(v)+\{1\}^{d-k}$.
\end{proof}

\subsection*{Discrete Brunn-Minkowski}In  \cite{GTcompressions}, Green and Tao noted that Brunn--Minkowski inequality can be used to derive lower bounds for discrete sumsets. We present a straightforward consequence of their arguments as the following lemma.

\begin{lemma}[Discrete Brunn--Minkowski] \label{lem:discbm}
    For finite, non-empty sets $X,Y\subset \Z^d$, we have
    \[|X+Y+\{0,1\}^d|^{1/d}\geq |X|^{1/d}+|Y|^{1/d}.\]
\end{lemma}

\begin{proof}
    Let $\mathcal{A} =  X + (0,1)^d$ and $\mathcal{B} = Y + (0,1)^d$ and let $\mu$ be the Lebesgue measure in $\mathbb{R}^d$. Brunn--Minkowski inequality implies that
    \[ \mu(\mathcal{A} + \mathcal{B})^{1/d}  \geq \mu(\mathcal{A})^{1/d} + \mu(\mathcal{B})^{1/d}. \]
    Observe that $\mu(\mathcal{A}) = |X|$ and $\mu(\mathcal{B}) = |Y|$ and 
    \[ \mu(\mathcal{A} + \mathcal{B}) = \mu(X + Y + (0,2)^d) = \mu(X + Y + \{0,1\}^d + (0,1)^d) = |X+Y + \{0,1\}^d|. \]
    Combining this with the preceding inequality finishes the proof.
\end{proof}

For each $u\in \{0,1\}^k$, define the variable $x_u=|A_S(u)|^{\frac{1}{d-k}}\geq 0$, and for each $u\in \{0,1,2\}^k$, define the variable $y_u=|B_S(u)|^{\frac{1}{d-k}}\geq 0$. From Lemma \ref{lem:fibersum} and Lemma \ref{lem:discbm}, we see that for each $u,v\in \{0,1\}^k$, we have 
\[ y_{u+v}\geq x_u+x_v. \]
Since $A$ and $A_S$ are downsets, we also get that $x_u\geq x_v$ if $v\succeq u$. Therefore, Theorem \ref{thm:Cd-main} will follow from the inequality
\[\sum_{u\in \{0,1,2\}^k} y_u^{d-k}\geq C_d \sum_{u\in \{0,1\}^k} x_u^{d-k},\]
as described in Theorem \ref{lem:mainineq}. Hence, our goal for the rest of the paper is to present our proof of Theorem \ref{lem:mainineq}.


\section{Preliminary manoeuvres: Averaging and Redistribution}

We begin by recording some useful notations. First, for vectors $u = (u_1, \dots, u_d)$ and $v = (v_1, \dots, v_d)$ in $\{0,1\}^d$, we define the Hamming distance $d(u,v)$ to be
\[ d(u,v) = \sum_{1 \leq i \leq d: \ u_i \neq v_i}  1 =  \sum_{1 \leq i \leq d} |u_i - v_i|. \]
Moreover, we define
\[ |u| = d(u,0) = u_1 + \dots + u_d. \]
Next, recall that
\[C_d \;=\;
  \begin{cases}
    2\bigl(\tfrac32\bigr)^{d-1}, & d \le 5,\\[4pt]
    2^{d/2+1} - 1, & d \ge 6 \text{ even},\\[4pt]
    3\cdot 2^{(d-1)/2} - \tfrac32, & d \ge 7 \text{ odd}.
  \end{cases}
\]

As a first sanity check, we note that when $k=0$, our analytic inequality \eqref{anineq} simply states that $y^d\geq C_d x^d$ whenever $y\geq 2x$. This is true since $C_d\leq 2^d$ for every $d \geq 1$.

Let us next derive some inequalities which hold for general $(k,d)$. We refer to the first of these as the Averaging Inequality (AI).

\begin{lemma}[AI] \label{lem:avgineq}
  Under the assumptions of \Cref{lem:mainineq}, we have
  \[\sum_{u\in \{0,1,2\}^k} y_u^{d-k} \geq (2(3/2)^k + 2^{d-k}-2) \sum_{u\in \{0,1\}^k} x_u^{d-k}.\]
\end{lemma}
\begin{proof}
For any $w \in \{0,1,2\}^k$, we define 
\[ r(w) = |\{(u,v) \in \{0,1\}^k \times \{0,1\}^k : u + v = w \}|. \]
If $w = (w_1, \dots, w_d)$, then
\begin{equation} \label{def4}
r(w) = \prod_{1 \leq i \leq d : w_i = 1} 2. 
\end{equation}

When $w\in \{0,2\}^k$, we observe that $y_{w}\geq 2x_{w/2}$. For all other $w\in \{0,1,2\}^k$, we can take the average over all decompositions $w=u+v$ with $u,v\in \{0,1\}^k$ to get that
\begin{align*}
    y_w^{d-k} \geq r(w)^{-1} \sum_{u+v=w} (x_u+x_v)^{d-k}  \geq r(w)^{-1} \sum_{u+v=w} (x_u^{d-k} +x_v^{d-k}  ).
\end{align*}
Summing this over all $w \in \{0,1,2\}^k$, we get that
\begin{align} \label{fi1}
    \sum_{w \in \{0,1,2\}^k} y_w^{d-k} 
    & \geq \sum_{u \in \{0,1\}^{k}} x_u^{d-k} (2^{d-k}  + 2\sum_{v \in \{0,1\}^k \setminus u } r(v+u)^{-1} )  \nonumber \\
    & = \sum_{u \in \{0,1\}^{k}} x_u^{d-k} (2^{d-k} - 2  + 2\sum_{v \in \{0,1\}^k } r(v+u)^{-1} ) .
\end{align}
Noting \eqref{def4}, we see that $r(u+v) = 2^{d(u,v)}$. Moreover, for any fixed $u \in \{0,1\}^k$ and $l \in \{0,1,\dots, k\}$, the number of $v \in \{0,1\}^k$ such that $d(u,v) = l$ is precisely ${ k \choose l}$. Thus, we see that for any $u \in \{0,1\}^k$, one has
\[  \sum_{v \in \{0,1\}^k} r(u+v)^{-1} =  \sum_{l=0}^k {k \choose l} 2^{-l} = (3/2)^k. \]
This combines with \eqref{fi1} to dispense the desired estimate.
\end{proof}

In particular, if $k=d-1$ then \nameref{lem:avgineq} gives the $k=d-1$ case of \Cref{lem:mainineq}:
\begin{equation}
    \sum_{u\in \{0,1,2\}^{d-1}} y_u \geq 2(3/2)^{d-1} \sum_{u\in \{0,1\}^{d-1}} x_u. \label{eq:k=d-1}
\end{equation}

Besides averaging over $y_w\geq x_u+x_v$ for all $u, v \in \{0,1\}^k$ with $u+v =w$, we can also use a carefully chosen  $(u,v)$ for every $w$. This will give us the Strong Inequality (SI).

\begin{lemma}[SI] \label{lem:strongineq}
  Under the assumptions of \Cref{lem:mainineq}, we have
  \[\sum_{u\in \{0,1,2\}^k} y_u^{d-k} \geq \sum_{s=0}^k D_s X_s,\]
  where for every $0 \leq s \leq k$, we define
  \begin{equation} \label{dsdef}
  D_s=2^{k-s}+(2^{d-k}-1)2^{s} \ \ \text{and} \ \ X_s=\sum_{u \in \{0,1\}^k : |u|=s} x_u^{d-k}. 
  \end{equation}
\end{lemma}
\begin{proof}
 For each $w\in \{0,1,2\}^k$, let $w^-,w^+\in \{0,1\}^k$ be defined as follows. For each $1\leq i\leq k$, let
\[w^-_i = \begin{cases}
    0 & \text{if } w_i=0 \text{ or } 1, \\
    1 & \text{if } w_i=2,\\
\end{cases} \ \ \text{and} \ \ 
w^+_i = \begin{cases}
    0 & \text{if } w_i=0, \\
    1 & \text{if } w_i=1 \text{ or } 2.
\end{cases}
\]
Then we have $w=w^-+w^+$ and $w^-\preceq w^+$. Thus, we have
\[y_w\geq x_{w^-}+x_{w^+}.\]
Using the inequality 
\[ (x_{w^-}+x_{w^+})^{d-k}\geq x_{w^-}^{d-k}+(2^{d-k}-1)x_{w^+}^{d-k} \]
and summing over all $w\in \{0,1,2\}^k$, we obtain the claimed estimate
\begin{align*}
    \sum_{u\in \{0,1,2\}^k} y_u^{d-k} &\geq \sum_{w\in \{0,1,2\}^k} x_{w^-}^{d-k} + (2^{d-k}-1)\sum_{w\in \{0,1,2\}^k} x_{w^+}^{d-k} \\
    &= \sum_{u\in \{0,1\}^k} (2^{k-|u|}+(2^{d-k}-1)2^{|u|})x_u^{d-k}. \qedhere
  \end{align*}
\end{proof}

For the rest of the paper, we set $D_s$ and $X_s$ as defined in \eqref{dsdef}. In particular, if $D_s\geq C_d$ for all $0\leq s\leq k$, then \Cref{lem:mainineq} follows from \nameref{lem:strongineq}. It turns out that there will be at most two $s$ for which $D_s<C_d$. If $D_s<C_d$ for some $s$ but $D_{s-1}>C_d$, we can get a little more juice out of \nameref{lem:strongineq} by using the compression properties to redistribute some weight from $D_{s-1}$ to $D_{s}$.
\begin{lemma}[Redistribution] \label{lem:redistribution}
  Under the assumptions of \Cref{lem:mainineq}, for $0\leq s<k$, we have
  \[
  X_s \geq \frac{s+1}{k-s} X_{s+1}.
  \]
\end{lemma}
\begin{proof}
  Since $x_u\geq x_v$ whenever $u\preceq v$, the average value of $x_u$ across all $u$ with $|u|=s$ is at least the average value of $x_u$ across all $u$ with $|u|=s+1$. In other words,
  \[\frac{1}{\binom{k}{s}}X_s \geq \frac{1}{\binom{k}{s+1}} X_{s+1},\]
  which gives the desired inequality. 
\end{proof}

\nameref{lem:strongineq} alone is enough to prove \Cref{lem:mainineq} for half of the pairs $(k,d)$, where roughly $k\leq d/2$. For the remaining half of the cases, we use \nameref{lem:strongineq} with \nameref{lem:redistribution}. These will prove \Cref{lem:mainineq} for almost all pairs $(k,d)$. The remaining few cases will be handled by a modification of \nameref{lem:strongineq} with \nameref{lem:redistribution}.

Observe that for $k=0$,  inequality  \eqref{anineq} follows from the fact that $2^d\geq C_d$, and for $k=d-1$,  inequality \eqref{anineq} follows from \eqref{eq:k=d-1} and the fact that $C_d\leq 2(3/2)^{d-1}$. Thus, we may assume that $1\leq k\leq d-2$, so the cases $d=1,2$ are already done.

\subsection{Case $d=3$} \label{cased3}

We have $C_d=\frac{9}{2}$ and we only need to consider the case $k=1$. In this case, \nameref{lem:avgineq} gives
\[\sum_{u\in \{0,1,2\}} y_u^{2}\geq 5 \sum_{u\in \{0,1\}} x_u^{2},\]
more than the required $\frac{9}{2}$.

\subsection{Case $d=4$}

We have $C_d=\frac{27}{4}$ and we need to consider the cases $k=1,2$. For $k=1$, \nameref{lem:avgineq} gives
\[\sum_{u\in \{0,1,2\}} y_u^{3}\geq 9 \sum_{u\in \{0,1\}} x_u^{3},\]
more than the required $\frac{27}{4}$. For $k=2$, \nameref{lem:strongineq} gives
\[\sum_{u\in \{0,1,2\}^2} y_u^{2}\geq \sum_{s=0}^2 (2^{2-s}+3\cdot 2^{s})X_s.\]
For each $s=0,1,2$, we have $2^{2-s}+3\cdot 2^{s} \geq \frac{27}{4}$, as required.

\subsection{Case $d=5$}
In this case, we prove that $C_d=\frac{81}{8}$. As before, we only need to consider the cases when $1 \leq k \leq d-2$. The case $k=1$ follows from \nameref{lem:avgineq} and the fact that 
\[ 2(3/2)^{k} +  2^{d-k} -2 = 17 > 10 + 1/8. \]
When $k=2$, we use \nameref{lem:strongineq} along with the fact that $ 2^{2- s} + 7 \cdot 2^s > 81/8$ for all $0 \leq s \leq 2$ to obtain the claimed estimate.  When $k=3$, we still apply \nameref{lem:strongineq} to get that
\[ \sum_{u \in \{0,1,2\}^3 } y_u^2 \geq \sum_{s = 0}^3 (2^{3 - s} + 3 \cdot 2^{s}) X_s=11X_0+10X_1+14X_2+24X_3.  \]
By \nameref{lem:redistribution}, we have $X_0\geq \frac{1}{3}X_1$, thus
\[ \sum_{u \in \{0,1,2\}^3 } y_u^2 \geq (10+1/4)X_0+(10+1/4)X_1+14X_2+24X_3.  \]
Each coefficient is at least $81/8$, so we are done.

\section{Proof of Theorem \ref{lem:mainineq}  when $d\geq 6$ is even}

Writing $m=d/2\geq 3$, we get that $C_d=2^{m+1}-1$. Recalling the discussion preceding \S\ref{cased3}, our aim turns to prove inequality \eqref{anineq} for every $1\leq k\leq 2m-2$. The table below summarises how each pair $(k,m)$ will be proven.
\begin{center}
  \begin{tabular}{c|p{10cm}}
  $(k,m)$ & Proof strategy \\ \hline
  $k\leq m$ & \nameref{lem:strongineq} \\
  $k=2m-2, m=4,5,6,7$ & \nameref{lem:strongineq} + \nameref{lem:redistribution} from $D_{k-m-2},D_{k-m-1}$ to $D_{k-m}$\\
  the rest & \nameref{lem:strongineq} + \nameref{lem:redistribution} from $D_{k-m-1}$ to $D_{k-m}$\\
  \end{tabular}
\end{center}

By \nameref{lem:strongineq}, we have
\[\sum_{u\in \{0,1,2\}^k} y_u^{d-k} \geq \sum_{u\in \{0,1\}^k} D_{|u|} x_u^{d-k},\]
where $D_s=2^{k-s}+(2^{d-k}-1)2^{s}$, and we are interested in the minimum value of $D_s$. Viewing $D_s$ as a function in the real variable $2^s$, we see that $D_s$ is convex and has a unique minimum when $2^s=\sqrt{\frac{2^k}{2^{d-k}-1}}$. Since $k\leq d-1$, $D_s$ is minimised at the integer $s=k-m$ or $k-m+1$, if these values are in the range $0\leq s\leq k$. 

If $k\leq m-1$, then $D_s$ is minimised at $s=0$, giving 
\[
  D_s\geq 2^k+2^{d-k}-1\geq 2^{d/2+1}-1=C_d.
\]

So we may assume that $k\geq m$. In this case, $D_s$ is minimised at $s=k-m$ or $k-m+1$, with values $D_{k-m}=2^{m+1}-2^{k-m}$ and $D_{k-m+1}=2^{m+1}+2^{m-1}-2^{k-m+1}\geq C_d$ when $k\leq d-2$. Thus, either $D_s$ is minimised at $s=k-m+1$ and we are done, or it is minimised at $s=k-m$. So from now on, we may assume that $D_s$ is minimised at $s=k-m$.

If $k=m$, then $D_s$ is minimised at $s=0$, and we have 
\[
  D_0=2^{m}+2^{m}-1=C_d.
\]

So we may assume that $k\geq m+1$. Unfortunately, $D_{k-m}<C_d$, so more work needs to be done. However, we see that $D_{k-m-1}=2^{m+1}+2^{m-1}-2^{k-m-1}\geq C_d$, so the only value of $D_s$ that is less than $C_d$ is $D_{k-m}$. We will now use the compression properties to distribute some of the weight from $D_{k-m-1}$ to $D_{k-m}$. Recall that $X_s=\sum_{u: |u|=s} x_u^{d-k}$. By \nameref{lem:redistribution}, we will be done if we can show that
\[\frac{k-m}{m+1}(D_{k-m-1}-C_d)+D_{k-m}\geq C_d,\]
or equivalently,
\[
\frac{k-m}{m+1}(2^{m-1}-2^{k-m-1}+1)\geq 2^{k-m}-1.
\]

Note that the function
\[f(k,m)=\frac{k-m}{m+1}(2^{m-1}-2^{k-m-1}+1) - 2^{k-m} + 1\]
is concave in $k$ for fixed $m$. Thus, its minimum occurs at one of the endpoints $k=m+1$ or $k=2m-2$. Checking for $k=m+1$, we have
\[f(m+1,m)=\frac{1}{m+1}(2^{m-1}-2^{0}+1) - 2^{1} + 1 = \frac{2^{m-1}}{m+1} - 1\geq 0\]
for $m\geq 3$. For $k=2m-2$, we have
\[f(2m-2,m)=\frac{m-2}{m+1}(2^{m-1}-2^{m-3}+1) - 2^{m-2} + 1 = \frac{2^{m-3}(m-8)+2m-1}{m+1},\]
which can be verified to be non-negative for $m=3$ and $m\geq 8$. For $k=2m-3$, we have
\[f(2m-3,m)=\frac{2^{m-4}(5m-23)+2m-2}{m+1}\geq 0\]
for $m\geq 4$. Thus, we have $f(k,m)\geq 0$ for all $m\geq 3$ and $m+1\leq k\leq 2m-2$, except for $m=4,5,6,7$ and $k=2m-2$.

For these remaining cases, we assume that $k=2m-2$ and $4\leq m\leq 7$. We will use \nameref{lem:redistribution} to further redistribute some weight from $D_{k-m-2}$ to $D_{k-m}$. By applying \nameref{lem:redistribution} twice, we have 
\[X_{k-m-2} \geq \frac{k-m-1}{m+2} X_{k-m-1}\geq \frac{(k-m-1)(k-m)}{(m+2)(m+1)} X_{k-m}.\]
Thus we will be done if we can show that
\[\frac{(k-m-1)(k-m)}{(m+2)(m+1)}(D_{k-m-2}-C_d)+\frac{k-m}{m+1}(D_{k-m-1}-C_d)+D_{k-m}\geq C_d,\]
or equivalently,
\[\frac{(m-3)(m-2)}{(m+2)(m+1)}(D_{m-4}-C_d)+\frac{m-2}{m+1}(D_{m-3}-C_d)+D_{m-2}\geq C_d,\]
or equivalently,
\[
  \frac{(m-3)(m-2)}{(m+2)(m+1)}(35\cdot 2^{m-4}+1) + \frac{m-2}{m+1}(3\cdot 2^{m-3}+1) \geq 2^{m-2}-1,
\]
which can be verified for $4\leq m\leq 7$ by direct computation. This completes the proof of \Cref{lem:mainineq} for even $d\geq 6$.

\section{Proof of Theorem \ref{lem:mainineq} when  $d\geq 7$ is odd}

Set $m=(d-1)/2$, then we have $C_d=3\cdot 2^{m}-3/2$. Our aim is to show Theorem \ref{lem:mainineq} for each $1\leq k\leq 2m-1$. The table below summarises how each pair $(k,m)$ will be proven.
\begin{center}
  \begin{tabular}{c|p{10cm}}
  $(k,m)$ & Proof strategy \\ \hline
  $k\leq m$ & \nameref{lem:strongineq} \\
  $k=m+1$ & Modified \nameref{lem:strongineq} (Lemma \ref{lem:strongineq2}) \\
  $(5,3),(6,4),(7,5)$ & Modified \nameref{lem:strongineq} (Lemma \ref{lem:strongineq3})\\
  $(7,4)$ & Modified \nameref{lem:strongineq} (Lemma \ref{lem:strongineq4}) + \nameref{lem:redistribution}\\
  the rest & \nameref{lem:strongineq} + \nameref{lem:redistribution}\\
  \end{tabular}
\end{center}

We will proceed just as in the case for even $d$. By \nameref{lem:strongineq}, we have
\[\sum_{u\in \{0,1,2\}^k} y_u^{d-k} \geq \sum_{s=0}^k D_s X_s,\]
where $D_s=2^{k-s}+(2^{d-k}-1)2^{s}$, and we are interested in the minimum value of $D_s$. Viewing $D_s$ as a function in the real variable $2^s$, we see that $D_s$ is convex and has a unique minimum when $2^s=\sqrt{\frac{2^k}{2^{d-k}-1}}$. Since $k\leq d-1$, $D_s$ is minimised at the integer $s=k-m-1$ or $k-m$, if these values are in the range $0\leq s\leq k$. If $k\leq m$, then $D_s$ is minimised at $s=0$, giving $D_s\geq 2^k+2^{d-k}-1\geq 2^{m}+2^{m+1}-1>C_d$. 

So we may assume that $k\geq m+1$. In this case, $D_s$ is minimised at $s=k-m-1$ or $k-m$, with values $D_{k-m-1}=3\cdot 2^{m}-2^{k-m-1}$ and $D_{k-m}=3\cdot 2^{m}-2^{k-m}\leq D_{k-m-1}$. Thus, $D_s$ is minimised at $s=k-m$. 

We now split into the cases $k=m+1$ and $k\geq m+2$.

\subsection{Case $k=m+1$}

Throughout the proof, we will be making several modifications to \nameref{lem:strongineq} to obtain different sets of coefficients. Below is one such modification.

\begin{lemma} \label{lem:strongineq2}
  Under the assumptions of \Cref{lem:mainineq}, the inequality
  \[\sum_{u\in \{0,1,2\}^k} y_u^{d-k} \geq \sum_{s=0}^k D_s' X_s\]
  holds for all $0 \leq k<d$ satisfying $2^{d-k}\geq k$,  where 
  \begin{itemize}
    \item $D_0'=D_0-1$,
    \item $D_1'=D_1+1$, 
    \item $D_2'=D_2-2^{d-k}+1$, and 
    \item $D_s'=D_s$ for $s\geq 3$.
  \end{itemize}
\end{lemma}
\begin{proof}
  Let $e_1,\ldots,e_k$ be the standard basis vectors of $\R^k$, and assume without loss of generality that $x_{e_1}\geq x_{e_2}\geq \cdots\geq x_{e_k}$. Recall that in the proof of \nameref{lem:strongineq}, we summed over all $w\in \{0,1,2\}^k$ the inequalities $y_w^{d-k}\geq x_{w^-}^{d-k}+(2^{d-k}-1)x_{w^+}^{d-k}$. We shall modify the inequality $y_{e_1+e_2}^{d-k}\geq x_{0}^{d-k}+(2^{d-k}-1)x_{e_1+e_2}^{d-k}$ to
  \[y_{e_1+e_2}^{d-k}\geq (x_{e_1}+x_{e_2})^{d-k}\geq x_{e_1}^{d-k}+(2^{d-k}-1)x_{e_2}^{d-k}.\]
  Then, \nameref{lem:strongineq} is modified to
  \begin{align*}
      \sum_{u\in \{0,1,2\}^k} y_u^{d-k} &\geq \sum_{s=0}^k D_sX_s + x_{e_1}^{d-k}+(2^{d-k}-1)x_{e_2}^{d-k} - x_{0}^{d-k}-(2^{d-k}-1)x_{e_1+e_2}^{d-k}\\
      &\geq \sum_{s=0}^k D_sX_s + \sum_{i=1}^k x_{e_i}^{d-k} - x_{0}^{d-k}-(2^{d-k}-1)x_{e_1+e_2}^{d-k},
  \end{align*}
  where the last inequality follows from the facts that $2^{d-k}-1\geq k-1$ and $x_{e_2}\geq x_{e_3},\ldots, x_{e_k}$. The lemma follows.
\end{proof}

We have $D_s=2^{m+1-s}+2^{m+s}-2^s$, and in particular, $D_0=3\cdot 2^{m}-1$, $D_1=3\cdot 2^{m}-2$ and $D_2=\frac{9}{2}\cdot 2^{m}-4$.  Taking the average of Lemma~\ref{lem:strongineq2} and \nameref{lem:strongineq}, we get
\begin{align}
    \sum_{u\in \{0,1,2\}^k} y_u^{m} &\geq \sum_{s=0}^k D_s''X_s, \label{eq:strongoddineq}
\end{align}
where $D_s'' = D_s$ for all $s\notin \{ 0,1,2\}$, and
\[D_0'' = D_0 - \frac{1}{2}, \quad D_1'' = D_1 + \frac{1}{2}, \quad D_2'' = D_2 - \frac{2^m-1}{2}.\]
Observe that $D_0''=D_1''=C_d$ and $D_2''=4\cdot 2^{m}-\frac{7}{2}\geq C_d$ for $m\geq 3$. Thus, for all $s$, we have $D_s''\geq C_d$, and so, \Cref{lem:mainineq} follows from \eqref{eq:strongoddineq} in this case.

\subsection{Case $k\geq m+2$} We have $D_{k-m}<D_{k-m-1}<C_d$. Note that $D_{k-m-2}=9\cdot 2^{m-1}-2^{k-m-2}$ and $D_{k-m+1}=9\cdot 2^{m-1}-2^{k-m+1}\leq D_{k-m-2}$. It can be checked that $D_{k-m+1}\geq C_d$ for $k\leq 2m-1$, thus we may use the convexity of $D_s$ to deduce that the only values of $D_s$ that are less than $C_d$ are $D_{k-m}$ and $D_{k-m-1}$. 

Just as before, we will redistribute some of the weight from $D_{k-m-2}$ to $D_{k-m-1}$ and $D_{k-m}$. By \nameref{lem:redistribution} twice, we have 
\[X_{k-m-2} \geq \frac{k-m-1}{m+2} X_{k-m-1}\geq \frac{(k-m-1)(k-m)}{(m+2)(m+1)} X_{k-m}.\]
Thus we will be done if we can show that
\[\frac{(m+2)(m+1)}{(k-m-1)(k-m)}(C_d-D_{k-m})+\frac{m+1}{k-m}(C_d-D_{k-m-1})\leq D_{k-m-2} - C_d,\]
or equivalently,
\begin{equation} \label{eq:doddweight}
\frac{(m+2)(m+1)}{(k-m-1)(k-m)}\left(2^{k-m}-\frac{3}{2}\right)+\frac{m+1}{k-m}\left(2^{k-m-1}-\frac{3}{2}\right)\leq 3\cdot 2^{m-1}-2^{k-m-2} + \frac{3}{2}.
\end{equation}
Writing
\begin{align*}
    f_1(k,m) &= (m+2)(m+1)\left(2^{k-m}-\frac{3}{2}\right)+(m+1)(k-m-1)\left(2^{k-m-1}-\frac{3}{2}\right)\\
    &\quad - (k-m-1)(k-m)\left(3\cdot 2^{m-1}-2^{k-m-2}+\frac{3}{2}\right)
\end{align*}
for all $k,m \in \mathbb{Z}$, we see that \eqref{eq:doddweight} follows for $k,m$ if $f_1(k,m)\leq 0$.

\begin{lemma} \label{kef}
The pairs $(k,m) \in \mathbb{Z}^2$ satisfying $m+2 \leq k \leq 2m-1$ and $m\geq 3$ for which $f_1(k,m)>0$ are 
\begin{equation} \label{eq:dodd-list1}
  (5,3),(6,4),(7,4),(7,5),(8,5),(9,5),(11,6),(13,7).
\end{equation}
\end{lemma}

\begin{proof}
Writing $x = k-m$, we see that $2 \leq x \leq m-1$ and $f_1(k,m) = g_m(x)$, where $g_m(x) = (m+2)(m+1)\left(2^x-\frac{3}{2}\right)+(m+1)(x-1)\left(2^{x-1}-\frac{3}{2}\right)  - (x-1)x\left(3\cdot 2^{m-1}-2^{x-2}+\frac{3}{2}\right)$. It's easy to see that
\begin{align*} 
g_m(x) &< (m+2)(m+1)2^x + (m+1)(x-1)2^{x-1} - (x-1)x(3\cdot 2^{m-1} - 2^{x-2}) \\
&= \left[(m+2)(m+1) + \frac{1}{2}(m+1)(x-1) + \frac{1}{4}(x-1)x\right]2^x - 3x(x-1)2^{m-1}.
\end{align*}
For $2 \leq x \leq m-1$ and $m \geq 3$, the expression in front of $2^x$ above is
$$
\leq (m+2)(m+1) + \frac{1}{2}(m+1)(m-2) + \frac{1}{4}(m-2)(m-1) = \frac{1}{4}(7m^2+7m+6) \leq \frac{1}{4}\cdot 10m^2.
$$
Thus $g_m(x) < h_m(x)$, where
$$
h_m(x) = \frac{5}{2}m^2\cdot 2^x - 3x(x-1)2^{m-1}.
$$
Suppose that $f_1(k,m) > 0$. Then $h_m(x) > 0$, and hence
$$
2^x > \frac{6}{5}m^{-2} x(x-1)2^{m-1} > m^{-2}2^{m-1}.
$$
Taking logarithms, we get
$$
x > m-1-2\log_2 m.
$$
First consider the case when $m \geq 16$. Then $\log_2m \leq m/3-1$, which implies that $ x \geq m/3+1$ and $x(x-1) \geq m^2/9$. Using $h_m(x) > 0$ again, we obtain
$$
2^x > \frac{6}{5}m^{-2}x(x-1)2^{m-1} \geq \frac{6}{5}m^{-2} \cdot \frac{1}{9}m^2 \cdot 2^{m-1} = \frac{2}{15}\cdot 2^{m-1} > 2^{m-4}
$$
which implies that $x \in \{m-1, m-2, m-3\}$. For $x \in \{m-2,m-3\}$, direct computations show that
\begin{align*}
h_m(m-2) &= \left[\frac{5}{4}m^2 - 3(m-2)(m-3)\right]2^{m-1} < 0,  \\
h_m(m-3) &= \left[\frac{5}{8}m^2 - 3(m-3)(m-4)\right]2^{m-1} < 0,
\end{align*}
when $m \geq 16$, a contradiction. For $x = m-1$, direct computations show that
$$
h_m(m-1) = \left[\frac{5}{2}m^2 - 3(m-1)(m-2)\right]2^{m-1} < 0
$$
when $m \geq 18$.

Hence, it suffices to check the cases when $(x,m) = (15,16), (16, 17)$ and when $m \leq 15$. At this point, one can run a computer program to verify that $g_m(x) >0$ precisely when $(k,m)$ is one of the 8 choices recorded in \eqref{eq:dodd-list1}.
\end{proof}

If $k\geq m+3$, we can further redistribute some weight from $D_{k-m-3}$ to $D_{k-m-2}$. By a further \nameref{lem:redistribution}, we will be done if we have
\begin{align*}
  &\frac{(m+2)(m+1)}{(k-m-1)(k-m)}(C_d-D_{k-m})+\frac{m+1}{k-m}(C_d-D_{k-m-1})\\
  &\leq D_{k-m-2} - C_d + \frac{k-m-2}{m+3}(D_{k-m-3}-C_d).
\end{align*}
It can be verified that among the pairs in \eqref{eq:dodd-list1}, the only ones with $k<m+3$ or which do not satisfy the above inequality are
\[
  (5,3),(6,4),(7,4),(7,5).
\]
At this point, we have fully squeezed all the juice out of \nameref{lem:strongineq} and \nameref{lem:redistribution}, so to continue with these four cases, we require a modification (or two) of \nameref{lem:strongineq}.

\begin{lemma} \label{lem:strongineq3}
  Under the assumptions of \Cref{lem:mainineq} with $k\geq 3$, we have
  \[\sum_{u\in \{0,1,2\}^k} y_u^{d-k} \geq \sum_{s=0}^k D_s' X_s,\]
  where 
  \begin{itemize}
    \item $D_0'=D_0-\binom{k}{2}-2^{d-k+1}+4$,
    \item $D_1'=D_1+1$, 
    \item $D_2'=D_2+2^{d-k}-1$, 
    \item $D_3'=D_3-2^{d-k}+1$, and 
    \item $D_s'=D_s$ for $s\geq 4$.
  \end{itemize}
\end{lemma}
\begin{proof}
  Let $e_1,\ldots,e_k$ be the standard basis vectors of $\R^k$, and assume without loss of generality that $x_{e_1}\geq x_{e_2}\geq \cdots\geq x_{e_k}$. Recall that in the proof of \nameref{lem:strongineq}, we summed over all $w\in \{0,1,2\}^k$ the inequalities $y_w^{d-k}\geq x_{w^-}^{d-k}+Rx_{w^+}^{d-k}$, where $R=2^{d-k}-1$ for simplicity. We shall modify the inequalities for $w=e_i+e_j+e_l$ for several triples $(i,j,l)$ into $y_{e_i+e_j+e_l}^{d-k}\geq (x_{e_i}+x_{e_j+e_l})^{d-k}\geq x_{e_i}^{d-k}+Rx_{e_j+e_l}^{d-k}$, where the last inequality follows from $x_{e_i}\geq x_{e_j+e_l}$, which holds if $i<\max(j,l)$. The set $T$ of triples $(i,j,l)$ which we will be modifying are the following.
  \begin{enumerate}
    \item $(1,2,3)$;
    \item $(i,i+1,l)$ for $1\leq i\leq l-2\leq k-2$;
    \item $(l-1,1,l)$ for $3\leq l\leq k$.
  \end{enumerate}
  Then $|T|=\binom{k}{2}-2$. It can be checked that all triples in $T$ satisfy $i<\max(j,l)$, and that the set $\{i,j,l\}$ for $(i,j,l)\in T$ are all distinct. Thus, after modifying \nameref{lem:strongineq}, we obtain
  \begin{align} \label{eq:modifiedstrong}
      \sum_{u\in \{0,1,2\}^k} y_u^{d-k} &\geq \sum_{s=0}^k D_s X_s + \sum_{(i,j,l)\in T} (x_{e_i}^{d-k}+Rx_{e_j+e_l}^{d-k} - x_{0}^{d-k} - Rx_{e_i+e_j+e_l}^{d-k}).
  \end{align}
  Among the triples in $T$, 
  \begin{enumerate}
    \item each $e_i+e_j+e_l$ appears at most once;
    \item each $e_i$ appears at least once, except for $e_k$ which does not appear;
    \item $e_1$ appears at least twice;
    \item each $e_j+e_l$ appears exactly once, except for $e_1+e_2$ and $e_1+e_3$ which do not appear.
  \end{enumerate}
  Using $x_{e_1}\geq x_{e_k}$ and $x_{0}\geq x_{e_1+e_2},x_{e_1+e_3}$, the right hand side sum of \eqref{eq:modifiedstrong} is at least
  \[X_1+RX_2-\paren{|T|+2R}X_0-RX_3.\]
  This gives the desired inequality.
\end{proof}

For the cases $(k,m)=(5,3),(6,4),(7,5)$, we will use a weighted average of \nameref{lem:strongineq} and Lemma~\ref{lem:strongineq3}. The final case $(k,m)=(7,4)$ will be done with yet another modification of SI with \nameref{lem:redistribution}.

\subsubsection{Case $(k,m)=(5,3)$}

\nameref{lem:strongineq} gives
\begin{equation} \label{eq:53strong}
  \sum_{u\in \{0,1,2\}^5} y_u^{2} \geq 35X_0+22X_1+20X_2+28X_3+50X_4+97X_5.
\end{equation}
Lemma~\ref{lem:strongineq3} gives
\begin{equation} \label{eq:53strong2}
  \sum_{u\in \{0,1,2\}^5} y_u^{2} \geq 21X_0+23X_1+23X_2+25X_3+50X_4+97X_5.
\end{equation}
Our aim is to get all coefficients at least $C_d=22.5$. Taking $1/6$ times \eqref{eq:53strong} and $5/6$ times \eqref{eq:53strong2}, we get
\[\sum_{u\in \{0,1,2\}^5} y_u^{2} \geq \frac{70}{3}X_0+\frac{137}{6}X_1+\frac{45}{2}X_2+\frac{51}{2}X_3+50X_4+97X_5,\]
as required.

\subsubsection{Case $(k,m)=(6,4)$}

\nameref{lem:strongineq} gives
\begin{equation} \label{eq:64strong}
  \sum_{u\in \{0,1,2\}^6} y_u^{3} \geq 71X_0+46X_1+44X_2+64X_3+\cdots.
\end{equation}
Lemma~\ref{lem:strongineq3} gives
\begin{equation} \label{eq:64strong2}
  \sum_{u\in \{0,1,2\}^6} y_u^{3} \geq 44X_0+47X_1+51X_2+57X_3+\cdots.
\end{equation}
Our aim is to get all coefficients at least $C_d=46.5$, and the terms in the $\cdots$ are already more than $C_d$. Taking the average of \eqref{eq:64strong} and \eqref{eq:64strong2} gives us what we want.

\subsubsection{Case $(k,m)=(7,5)$}
\nameref{lem:strongineq} gives
\begin{equation} \label{eq:75strong}
  \sum_{u\in \{0,1,2\}^7} y_u^{4} \geq 143X_0+94X_1+92X_2+136X_3+\cdots.
\end{equation}
Lemma \ref{lem:strongineq3} gives
\begin{equation} \label{eq:75strong2}
  \sum_{u\in \{0,1,2\}^7} y_u^{4} \geq 94X_0+95X_1+107X_2+121X_3+\cdots.
\end{equation}
Our aim is to get all coefficients at least $C_d=94.5$, so taking the average of \eqref{eq:75strong} and \eqref{eq:75strong2} gives us the desired lower bound.

\subsubsection{Case $(k,m)=(7,4)$}

In this case, \nameref{lem:redistribution} is barely not enough to get all coefficients at least $C_d=46.5$. The following modification of \nameref{lem:strongineq} gives us just enough push to allow us to use \nameref{lem:redistribution}.
\begin{lemma} \label{lem:strongineq4}
  Let $3\leq t\leq k$. Under the assumptions of Theorem~\ref{lem:mainineq}, we have
  \[\sum_{u\in \{0,1,2\}^k} y_u^{d-k} \geq \sum_{s=0}^k D_s' X_s,\]
  where 
  \begin{itemize}
    \item $D_0'=D_0-\binom{k}{t}$,
    \item $D_1'=D_1+\frac{1}{k}\binom{k}{t}$, 
    \item $D_{t-1}'=D_{t-1}+\binom{k}{t}/\binom{k}{t-1}$, 
    \item $D_t'=D_t-2^{d-k}+1$, and 
    \item $D_s'=D_s$ for all other $s$.
  \end{itemize}
\end{lemma}
\begin{proof}
  Recall that in the proof of \nameref{lem:strongineq}, we summed over all $w\in \{0,1,2\}^k$ the inequalities $y_w^{d-k}\geq x_{w^-}^{d-k}+Rx_{w^+}^{d-k}$, where $R=2^{d-k}-1$. When $w\in \{0,1\}^k$, this meant that we used the inequality $y_w^{d-k} \geq x_0^{d-k} + R x_w^{d-k}$. We will modify this now, and so,  for any $w \in \{0,1\}^k$ such that $|w|=t$, let $u\in \{0,1\}^k$ be any vector with $|u|=1$ and $u\preceq w$. We shall modify the inequality for $w$ from $y_w^{d-k}\geq x_0^{d-k}+Rx_w^{d-k}$ to
  \[y_w^{d-k}\geq (x_u + x_{w-u})^{d-k} \geq x_u^{d-k} + x_{w-u}^{d-k}.\]
  Taking the average over all such $u$ for each $w$, we get the desired inequality.
\end{proof}

Back to the case $(k,m)=(7,4)$, Lemma~\ref{lem:strongineq4} with $t=4$ gives
\[\sum_{u\in \{0,1,2\}^7} y_u^{2} \geq 96X_0+75X_1+44X_2+41X_3+53X_4+\cdots.\]
Our aim is to get all coefficients at least $C_d=46.5$. \nameref{lem:redistribution} gives $X_0\geq \frac{1}{7}X_1$, $X_1\geq \frac{1}{3}X_2$ and $X_2\geq \frac{3}{5}X_3$, and now we have enough weights to redistribute from $X_0,X_1$ to $X_2,X_3$ to get all coefficients at least $46.5$.

\section{Concluding Remarks}

Our work suggests several natural further directions. One appealing avenue is to extend our sharp bounds for sets 
\(A \subset \mathbb{Z}^d\) containing the discrete hypercube \(\{0,1\}^d\) to settings when \(A\) contains other types of high-dimensional structures. One potential example of this is the \emph{permutohedron}
\[
P_d := \operatorname{conv}\{(\sigma(1),\dots,\sigma(d)) \in \R^d: \sigma \in S_d\} \subset \mathbb{R}^d.
\]

The permutohedron $P_{d}$ lives inside the hyperplane $\left\{x \in \mathbb{R}^{d}:\ x_1+\ldots+x_d = \binom{d+1}{2}\right\}$, and has a number of remarkable properties (see for example \cite{Postnikov} or \cite{Stanley}). 
For instance, when $d$ is even, the vertex set of $P_d$ contains $d/2$-dimensional affine cubes. Indeed, consider the $d/2$ disjoint adjacent transpositions
\[
t_j = (2j-1\;\;2j) \in S_{d}, \qquad j = 1,\dots,d/2.
\]
These transpositions commute and each of them has order two. Thereby, for each 
$\varepsilon = (\varepsilon_1,\dots,\varepsilon_{d/2})\in\{0,1\}^{d/2}$, we can define the permutation
\[
\sigma_\varepsilon := t_1^{\varepsilon_1} \circ t_2^{\varepsilon_2} \cdots \circ t_{d/2}^{\varepsilon_{d/2}} \in S_{d}.
\]
Then the $2^{d/2}$ vertices of $P_{d}$ defined by
\[
\Bigl\{\,\bigl(\sigma_\varepsilon(1),\dots,\sigma_\varepsilon(d)\bigr)
   : \varepsilon\in\{0,1\}^{d/2} \Bigr\}
\]
form an affine copy of the ${d/2}$-dimensional cube $\{0,1\}^{d/2}$. Consequently, if a finite set \(A \subset \mathbb{Z}^d\) contains the vertex set of the permutohedron, then \(A\) automatically contains 
an affine copy of \(\{0,1\}^{\lfloor d/2 \rfloor}\).  
Our Theorem \ref{thm:Cd-main} therefore implies
\begin{equation} \label{cubeinP}
|A+A|\geq C_{d/2}(|A|-2^{d/2})+3^{d/2}.
\end{equation}

On the other hand, it is also not difficult to check that any $k$-dimensional affine hypercube inside the vertex set of $P_d$ has dimension at most $d/2$. Therefore, it is an interesting problem to leverage the additional geometric structure of the permutohedron to improve upon \eqref{cubeinP}, and then to more precisely determine the largest constant $K_d > 0$ such that
\[
|A+A| \;\ge\; K_{d}\,|A| - o_{d}(|A|),
\]
for all sets \(A \subset \mathbb{Z}^d\) containing the vertices of $P_d$. In this direction, we note that the permutohedron admits a coordinate compression (via its Lehmer--code representation) to the full down-set
\[
[d]\times[d-1]\times\cdots\times[1],
\]
so it suffices to determine the minimum size of $|A+A|$ in terms of $|A|$ for sets \(A \subset \mathbb{Z}^d\) containing the set $[d]\times[d-1]\times\cdots\times[1]$. This fact already immediately enables one to replace $C_{d/2}$ with $C_{d-1}$ in \eqref{cubeinP}, but we suspect that $K_{d}$ is still significantly larger than $C_{d-1}$. 

Another related problem is to determine, for every fixed $k \geq 1$, the largest constant \(C_{d,k} >0\) such that
\[
|A+A| \;\ge\; C_{d,k}\,|A|-o_{d,k}(|A|),
\]
whenever \(A \subset \mathbb{Z}^d\) contains the full discrete box \(\{0,1,\dots,k\}^d\), or the full discrete simplex 
\[ \setcond{(x_1,\ldots,x_d)\in \mathbb{Z}^d_{\geq 0}}{x_1+\cdots+x_d\leq k} \]
with larger side lengths.

\end{document}